\documentclass[11pt]{amsart}
\usepackage{epsfig,amssymb}

\theoremstyle{plain}
\newtheorem{theorem}{Theorem}[section]
\newtheorem{proposition}[theorem]{Proposition}

\newtheorem{remark}[theorem]{Remark}

\newtheorem{question}[theorem]{Question}

\title{\bf On the union stabilization of \\two Heegaard splittings}

\author{Jung Hoon Lee}

\address{\small School of Mathematics, KIAS\\
207-43, Cheongnyangni 2-dong, Dongdaemun-gu\\
Seoul, Korea\\} \email{jhlee@kias.re.kr}

\date{}

\begin{document}



\begin{abstract}
Let two Heegaard splittings $V_1\cup W_1$ and $V_2\cup W_2$ of a
$3$-manifold $M$ be given. We consider the {\it union
stabilization} $M=V\cup W$ which is a common stabilization of
$V_1\cup W_1$ and $V_2\cup W_2$ having the property that
$V=V_1\cup V_2$. We show that any two Heegaard splittings of a
$3$-manifold have a union stabilization. We also give some
examples with numerical bounds on the minimal genus of union
stabilization. On the other hand, we give an example of a
candidate for which the minimal genus of union stabilization is
strictly larger than the usual stable genus --- the minimal genus
of common stabilization.
\end{abstract}

\maketitle

\section{Introduction}

 A {\it Heegaard splitting} $M=H_1\cup_S H_2$ of a $3$-manifold $M$
 is a decomposition of $M$ into two handlebodies $H_1$ and $H_2$
 and it is well known that every compact $3$-manifold admits
 Heegaard splittings.

 A {\it stabilization} of a Heegaard
 splitting $H_1\cup_S H_2$ is an operation that results in
 a new Heegaard splitting $H'_1\cup_{S'} H_2'$, with the genus
 increased by one. Here $H'_1$ is obtained by adding trivial $1$-handle
 to $H_1$ whose core is $\partial$-parallel in $H_2$ and $H'_2$ is
 obtained by removing the $1$-handle from $H_2$.

 When two Heegaard splittings $V_1\cup W_1$ and $V_2\cup W_2$ of a
 $3$-manifold are given, they become isotopic after a sequence of
 stabilizations \cite{Singer}. The minimal genus among the common
 stabilizations is called the {\it stable genus}. Now we introduce
 a new concept slightly stronger than the common stabilization.

 A Heegaard splitting $V\cup W$ is a {\it union stabilization}
 of $V_1\cup W_1$ and $V_2\cup W_2$ if
 \begin{itemize}
\item{$V\cup W$ is a common stabilization
 of $V_1\cup W_1$ and $V_2\cup W_2$.}
 \item{$V=V_1\cup V_2$}
 \end{itemize}

 In section $2$, we show that any two Heegaard splittings have a union
 stabilization, similar to the results of \cite{Singer}. The
 minimal genus among the union stabilizations is called the {\it union
 genus}.

 \begin{remark}
 $(1)$ From the definition of the union stabilization, it follows
 that $W=W_1\cap W_2$.\\
 $(2)$ The union stabilization and the union genus can be defined
similarly for Heegaard splittings of $3$-manifolds with boundary.
\end{remark}

 When two Heegaard splittings of a $3$-manifold are given, it has
 been conjectured that they become isotopic after a single
 stabilization of the larger genus one \cite{Kirby}. However,
 recently there are results that this is not true and there are
 examples that require as many stabilizations as the genus of the
 Heegaard splittings \cite{HTT}, \cite{Johnson1}, \cite{Johnson2},
 \cite{Bachman}. This implies that the union genus also can
 be large enough. In section $3$, we give some examples with
 bounds on the union genus.

 In section $4$, we give an example of a candidate for which the
 union genus is strictly larger than the stable genus implying
 that a common stabilization of two Heegaard splittings cannot be
 obtained by the union of both handlebodies.

\section{Existence of the union stabilization}

We begin with a simple case where two Heegaard splittings are
isotopic.

\begin{proposition}Suppose $V_1\cup W_1$ and $V_2\cup W_2$ are
isotopic Heegaard splittings of a $3$-manifold. Then there is a
union stabilization with genus increased by one.
\end{proposition}

\begin{proof}
The two handlebodies of Heegaard splittings are yellow and blue
handlebodies respectively (Fig. 1). Let some parts of each
handlebodies coincide and it is colored green. We can see that the
disk $D$ serves as a stabilizing disk for both.
\end{proof}

\begin{figure}[h]
    \centerline{\includegraphics[width=7cm]{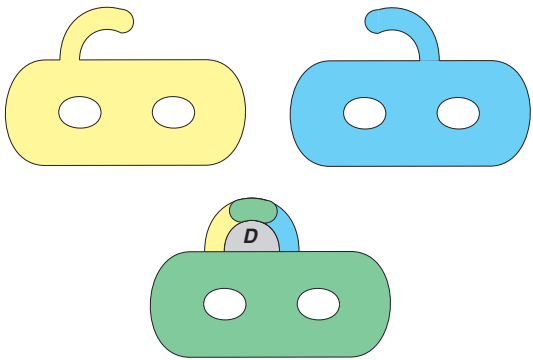}}
    \caption{Yellow $+$ Blue $=$ Green}
\end{figure}

A {\it spine} $\Sigma_H$ of a handlebody $H$ is a finite graph in
$H$ where $H$ defomation retracts to $\Sigma_H$.

 The following is a result similar to \cite{Singer}.

\begin{theorem}
Let $V_1\cup_{S_1} W_1$ and $V_2\cup_{S_2} W_2$ be two Heegaard
splittings of a $3$-manifold $M$. Then there exists a union
stabilization of $V_1\cup_{S_1} W_1$ and $V_2\cup_{S_2} W_2$.
\end{theorem}

\begin{proof}
Take $V_1$ and $V_2$ as a very thin neighborhood of $\Sigma_{V_1}$
and $\Sigma_{V_2}$ respectively. So we may assume that $V_1\cap
V_2=\emptyset$. Hence $V_2$ is in $W_1$. Furthermore, we assume
that $V_2$ is an embedded graph in $W_1$. Note that genus $g$
handlebody is homeomorphic to $g$\,-punctured disk $\times$ $I$.
We consider the projection of $V_2$ to the $g$\,-punctured disk
and the diagram of $V_2$ with over and under information.

If we add sufficiently many tunnels to $V_2$ (at least, as many
tunnels as the number of crossings of the diagram), the exterior
in $W_1$ would be a handlebody. Connect each tunnel to $\partial
W_1$ with an arc so that the exterior of the union of $V_2$ and
tunnels and arcs is still a handlebody. However, adding tunnels
and arcs is equivalent to the isotopy of some parts of $V_2$ along
the tunnels and arcs and let the parts coincide with some parts of
$V_1$. Fig. 2 shows the isotopy. Then we can see that $(V_1\cup
V_2)\cup cl((V_1\cup V_2)^c)$ is a union stabilization since
$V_1\cap V_2$ is a collection of $3$-balls and by the uniqueness
of Heegaard splittings of a handlebody \cite{ST}.
\end{proof}

\begin{figure}[h]
    \centerline{\includegraphics[width=8cm]{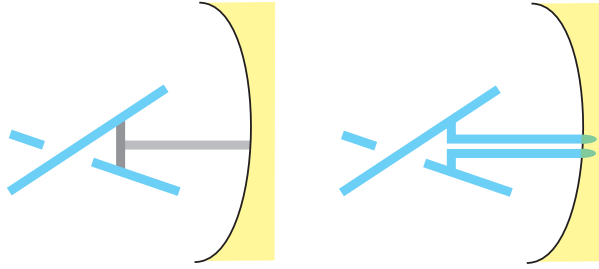}}
    \caption{Isotopy of $V_2$ and $V_1\cap V_2$ $=$ green $3$-balls}
\end{figure}

\section{Examples}
In this section we show some examples with numerical bounds on the
union genus. First example is that the spine of the handlebody of
one Heegaard splitting is on the other Heegaard surface.
 For a Heegaard splitting
$V\cup_S W$, let $g(S)$ denote the genus of $S$.

\begin{proposition}
Let $V_1\cup_{S_1} W_1$ and $V_2\cup_{S_2} W_2$ be two Heegaard
splittings of a $3$-manifold $M$ such that $\Sigma_{V_2}$ is in
$S_1$. Then the union genus of $V_1\cup_{S_1} W_1$ and
$V_2\cup_{S_2} W_2$ is less than or equal to $g(S_1)+g(S_2)$.
\end{proposition}

\begin{proof}
Take $V_2$ as s a very thin neighborhood of $\Sigma_{V_2}$. Push
$V_2$ slightly into $W_1$ so that there is an annulus of
parallelism between $\partial V_2$ and $\partial V_1$. Connect
$V_1$ and $V_2$ by an essential arc of the annulus. Then the
result is a $g(S_2)$-times stabilization of $V_1\cup_{S_1} W_1$.
Connecting $V_1$ and $V_2$ by an arc is equivalent to making some
parts of $V_1$ and $V_2$ coincide as in the proof of Theorem 2.2.
Then since $V_1\cap V_2$ is a $3$-ball and by the uniqueness of
Heegaard splittings of a handlebody \cite{ST}, it is also a
$g(S_1)$-times stabilization of $V_2\cup_{S_2} W_2$. Hence the
union genus of $V_1\cup_{S_1} W_1$ and $V_2\cup_{S_2} W_2$ is less
than or equal to $g(S_1)+g(S_2)$.
\end{proof}

The next example is the $(g,n)$ presentation of a knot $K$ in
$S^3$. Let $V_1\cup V_2$ be a decomposition of $S^3$ into two
standardly embedded genus $g$ handlebodies $V_1$ and $V_2$.
Suppose that $V_i\cap K$ ($i=1,2$) are trivial $n$-string tangles.
Then we call $(V_1\cap K, V_2\cap K)$ the genus $g$, $n$-bridge
presentation of $K$, or $(g,n)$ presentation for short. Clearly
such a knot $K$ has tunnel number less than or equal to $g+n-1$ as
Fig.3 illustrates.

\begin{figure}[h]
    \centerline{\includegraphics[width=5cm]{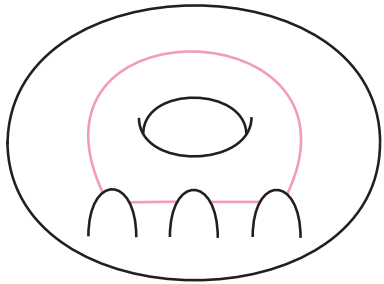}}
    \caption{$(g,n)$ presentation of a knot with a tunnel system ($g=1$, $n=3$)}
\end{figure}

\begin{proposition}
Suppose a knot $K$ with a $(g,n)$ presentation has tunnel number
$g+n-1$. Let $\mathcal{T}_i$ be a tunnel system in $V_i$ as in
Fig. 3 $(i=1,2)$. Then the two Heegaard splittings induced by
$\mathcal{T}_1$ and $\mathcal{T}_2$ has union genus less than or
equal to $2g+2n-1$.
\end{proposition}

\begin{proof}
It suffices to show that the exterior in $S^3$ of the union of $K$
and the tunnels in $\mathcal{T}_1$ and $\mathcal{T}_2$ is a genus
$2g+2n-1$ handlebody. The exterior $X_1\subset V_1$ of the union
of the trivial $n$-string tangle and $\mathcal{T}_1$ is a genus
$2g+2n-1$ handlebody. The exterior $X_2\subset V_2$ of the union
of the trivial $n$-string tangle and $\mathcal{T}_2$ is
homeomorphic to a ($2n$-punctured genus $g$ surface) $\times$ $I$.
Since $X_1\cap X_2$ is a $2n$-punctured genus $g$ surface, the
total exterior $X_1\cup X_2$ is homeomorphic to a genus $2g+2n-1$
handlebody.
\end{proof}

\section{Towards an example that union genus $>$ stable genus}

In this section, we construct an example of a candidate for which
the union genus is strictly larger than the stable genus.

Consider a $2$-bridge knot and two unknotting tunnels as in the
Fig. 4, which are put on the bridge sphere \cite{Kobayashi}.

\begin{figure}[h]
    \centerline{\includegraphics[width=3.5cm]{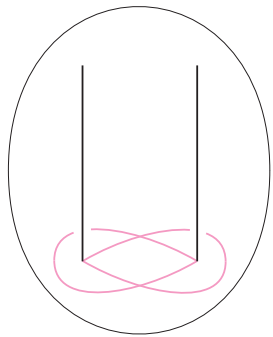}}
    \caption{Two unknotting tunnels of a $2$-bridge knot}
\end{figure}

The two unknotting tunnels intersect in two points. By some
perturbation, we can make them disjoint. Then there exist some
knot, for example a $2$-bridge knot $6_3$ in knot table, such that
the four cases of perturbations as in Fig. 5 do not give genus
three handlebodies. (For each case, we can check the fundamental
group of the exterior using the Wirtinger presentation.)

\begin{figure}[h]
    \centerline{\includegraphics[width=6cm]{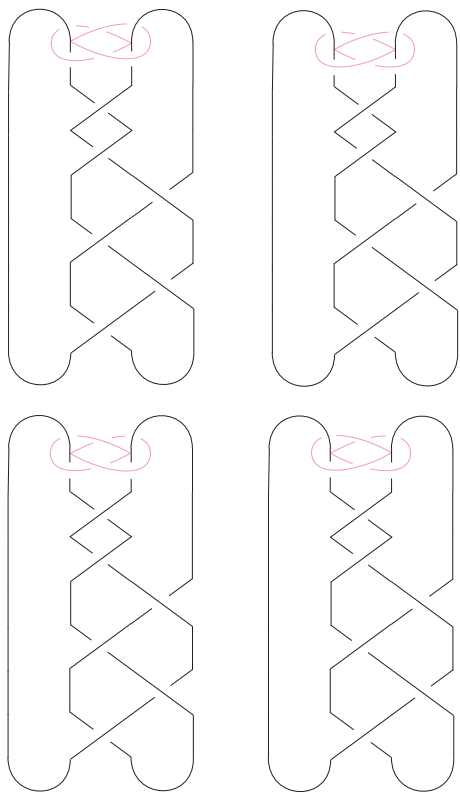}}
    \caption{These do not give genus three Heegaard splittings.}
\end{figure}

 We guess that the two Heegaard splittings of the
$2$-bridge knot $6_3$ induced by the unknotting tunnels in Fig. 4
and Fig. 5 have union genus four.

\begin{remark}
We can see that the stable genus of above example is three. The
common stabilization is isotopic to the stabilization of the genus
two Heegaard splitting induced by the tunnel connecting the two
strands of the knot (Fig. 6).
\end{remark}

\begin{figure}[h]
    \centerline{\includegraphics[width=7cm]{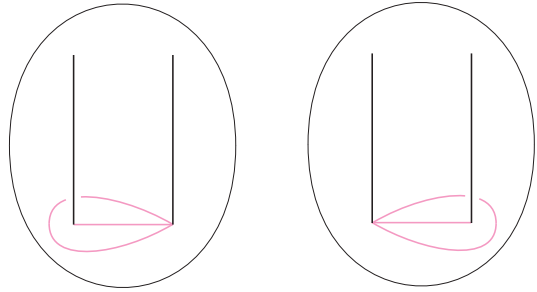}}
    \caption{The stable genus is three.}
\end{figure}

So finally we have the following question.

\begin{question}
There exists two Heegaard splittings of a $3$-manifold that the
union genus is strictly larger than the stable genus.
\end{question}

\end{document}